\documentclass[12pt]{amsart}
\usepackage{amsfonts,graphics,amsmath,amsthm,amsfonts,amscd,amssymb,amsmath, enumerate,latexsym,multicol,mathrsfs}
\usepackage{graphicx, hhline}
\usepackage[all]{xy}
\usepackage[usenames]{color}
\usepackage[colorlinks,
linkcolor=blue,
anchorcolor=black,
citecolor=green
]{hyperref}
\usepackage{fancyhdr}
\usepackage{epsfig,url}
\usepackage[top=30truemm,bottom=30truemm,left=30truemm,right=30truemm]{geometry}

\hypersetup{colorlinks=true}

\title{Existence of canonical models for Kawamata log terminal pairs}
\author{ Zhengyu Hu}
\date{2020/04/06}

\address{Mathematical Sciences Research Center, 
	Chongqing University of Technology, No.69 Hongguang Avenue, Chongqing, 400054, China}
\email{zhengyuhu16@gmail.com}

\newcommand{\Supp}[0]{{\operatorname{Supp}}}

\DeclareMathOperator{\CDiv}{CDiv}

\DeclareMathOperator{\ex}{Ex}

\newtheorem{thm}{Theorem}[section]

\newtheorem{lem}[thm]{Lemma}

\theoremstyle{definition}
\newtheorem{defn}[thm]{Definition}
\newtheorem{rem}[thm]{Remark}

\newtheorem*{claim*}{Claim}

\newcommand{\K}{\mathbb K}

\newcommand{\Q}{\mathbb Q}
\newcommand{\R}{\mathbb R}

\newcommand{\rddown}[1]{\left\lfloor{#1}\right\rfloor} 
\begin{document}

\maketitle

\begin{abstract}
We prove that a Kawamata log terminal pair has the canonical model.
\end{abstract}


\section{Introduction}\label{sec1}
We work over an algebraically closed field of characteristic zero. \\

Our main result is the existence of canonical models for Kawamata log terminal pairs.
\begin{thm}\label{thm-canonical-model}
	Let $(X/Z,B)$ be a Kawamata log terminal pair with the Kodaira dimension $\kappa_\iota(X/Z,K_X+B) \ge 0$. Then, $(X/Z,B)$ has the canonical model.
\end{thm}

If $B$ is a $\Q$-divisor, then Theorem \ref{thm-canonical-model} is \cite[Corollary 1.1.2]{bchm}. In this paper, we prove it for the general case. The idea of proof is to reduce Theorem \ref{thm-canonical-model} to \cite[Theorem 1.2]{bchm}, by a canonical bundle formula of Fujino-Mori type for $\R$-divisors (cf. \cite{fujino-mori}).

\begin{thm}\label{thm-canonical-bundle-formula}
	Let $f:X \to Y$ be a contraction of normal varieties and $(X,B)$ be a klt pair such that $\kappa_{\iota}(X/Y,K_X+B)=0$. Then, there exists a commutative diagram
	$$
	\xymatrix{
		(X',B') \ar[d]_{f'} \ar[r]^{\pi}   &  (X,B)\ar[d]^{f}\  &\\
		Y' \ar[r]^{\phi} &    Y } 
	$$  
	which consists of birational models $\pi:X' \to X$ and $\phi:Y' \to Y$, such that:
	\begin{enumerate}
		\item $K_{X'}+B'=\pi^*(K_X+B)+E$ where $E$ is exceptional$/X$ and $B',E\ge 0$ have no common components.
		
		\item $K_{X'}+B' \sim_{\mathbb{R}} f'^*(K_{Y'}+B_{Y'}+M_{Y'})+R$ where $R \ge 0$ and $(Y',B_{Y'}+M_{Y'})$ is a g-klt generalised pair with the moduli b-divisor $\mathbf{M}$.
		
		\item $\kappa(X'/Y',R^h)=0$ and $R^v$ is very exceptional$/Y'$, where $R^h$ (resp. $R^v$) denotes the horizontal (resp. vertical) part over $Y'$.
	\end{enumerate}
\end{thm}
One can easily generalise the above theorem to log canonical pairs. See Remark \ref{rem-lc-formula}.

\section{Preliminaries}\label{sec2}
In this section we collect definitions and some important results. Throughout this paper all varieties are quasi-projective over a fixed algebraically closed field of characteristic zero and a divisor refers to an $\R$-Weil divisor unless stated otherwise. 

\subsection{Notations and definitions}\label{sec2-defn}
We collect some notations and definitions. We use standard definitions of Kawamata log terminal (klt, for short) pair and sub-klt pair (for example, see \cite[Section 2.1]{hu2}).

\noindent \textbf{Contractions.}
In this paper a \emph{contraction} refers to a proper morphism $f\colon X\to Y$ of varieties 
such that $f_*\mathcal{O}_X=\mathcal{O}_Y$. In particular, $f$ has connected fibres. Moreover, 
if $X$ is normal, then $Y$ is also normal. 
A birational map $\pi: X \dashrightarrow Y $ is a \emph{birational contraction} if the inverse of $\pi$ does not contract divisors. Note that $\pi$ is not necessarily a morphism unless stated otherwise. 

\noindent \textbf{Very exceptional divisors.}
Let $f : X \rightarrow Y$ be a
dominant morphism from a normal variety to a variety, $D$ a divisor on $X$, and $Z \subset X$ a closed
subset. We say $Z$ is \emph{horizontal} over $Y$ if $f(Z)$ dominates $Y$, and we say $Z$ is \emph{vertical} over $Y$ if $f(Z)$ is a proper subset of $Y$. 

Suppose $f$ is a contraction of normal varieties. Recall that a divisor $D$ is \emph{very exceptional}$/Y$ if $D$ is vertical$/Y$ and for any prime divisor
$P$ on Y there is a prime divisor $Q$ on $X$ which is not a component of $D$ but
$f(Q) = P$, i.e. over the generic point of $P$ we have $\mathrm{Supp} f^\ast P \nsubseteq \mathrm{Supp} D$.

If $\mathrm{codim} f(D) \ge 2$, then $D$ is very exceptional. In this case we say $D$ is \emph{$f$-exceptional}.


\noindent \textbf{Generalised pairs.}
For the basic theory of generalised polarised pairs (generalised pairs for short) we refer to \cite[Section 4]{birkarzhang}.
Below we recall some of the main notions and discuss some basic properties.

A \emph{generalised sub-pair} consists of 
\begin{itemize}
	\item a normal variety $X$ equipped with a proper
	morphism $X\to Z$, 
	
	\item an $\R$-divisor $B$ on $X$, and 
	
	\item a b-$\R$-Cartier  b-divisor over $X$ represented 
	by some projective birational morphism $\overline{X} \overset{\phi}\to X$ and $\R$-Cartier divisor
	$\overline{M}$ on $X$ such that $\overline{M}$ is nef$/Z$ and $K_{X}+B+M$ is $\R$-Cartier,
	where $M := \phi_*\overline{M}$.
\end{itemize}
A generalised sub-pair is a \emph{generalised pair} if $B$ is effective. We usually refer to the sub-pair by saying $(X/Z,B+M)$ is a generalised sub-pair \emph{with 
data} $\overline{M}$ or \emph{with the moduli b-divisor} $\mathbf{M}$, where $\mathbf{M}$ is represented by $\overline{M}$. We will use standard definitions of b-divisors, generalised singularities and log minimal models (for example, see \cite[Section 2.1]{hu2}).

\subsection{Iitaka dimension and Iitaka fibration}\label{subsec-iitaka-dim}
In this subsection we introduce the notion of invariant Iitaka dimension and invariant Iitaka fibration. 

Recall the following definitions of Iitaka dimension, which is a birational invariant integer given by the growth of the quantity of sections. 
\begin{defn}[Invariant Iitaka dimension]\label{defn--inv-iitaka-dim}
	Let $X$ be a normal projective variety, and $D$ be an $\mathbb{R}$-Cartier divisor $D$ on $X$. 
	We define the {\em invariant  Iitaka dimension} of $D$, denoted by $\kappa_{\iota}(X,D)$, as follows (see also \cite[Definition 2.5.5]{fujino-book}):  
	If there is an $\mathbb{R}$-divisor $E\geq 0$ such that $D\sim_{\mathbb{R}}E$, set $\kappa_{\iota}(X,D)=\kappa(X,E)$. 
	Here, the right hand side is the usual Iitaka dimension of $E$. 
	Otherwise, we set $\kappa_{\iota}(X,D)=-\infty$. 
	We can check that $\kappa_{\iota}(X,D)$ is well-defined, i.e., when there is $E\geq 0$ such that $D\sim_{\mathbb{R}}E$, the invariant Iitaka dimension $\kappa_{\iota}(X,D)$ does not depend on the choice of $E$. 
	By definition, we have $\kappa_{\iota}(X,D)\geq0$ if and only if $D$ is $\mathbb{R}$-linearly equivalent to an effective $\mathbb{R}$-divisor. 
	
	Let $X\to Z$ be a projective morphism from a normal variety to a variety, and let $D$ be an $\mathbb{R}$-Cartier divisor on $X$. 
	Then the {\em relative invariant Iitaka dimension} of $D$, denoted by $\kappa_{\iota}(X/Z,D)$, is defined by $\kappa_{\iota}(X/Z,D)=\kappa_{\iota}(X,D|_{F})$, where $F$ is a very general fibre (i.e. the fibre over a very general point) of the Stein factorisation of $X\to Z$.
	Note that the value $\kappa_{\iota}(X,D|_{F})$ does not depend on the choice of $F$ (see \cite[Lemma 2.10]{hashizumehu}). 
\end{defn}

For basic properties of the invariant Iitaka dimension, we refer to \cite[Remark 2.8]{hashizumehu}. 

\begin{defn}[Invariant Iitaka fibration]
	Let $X$ be a normal variety projective over $Z$, and $D$ be an $\mathbb{R}$-Cartier divisor on $X$ with $\kappa_{\iota}(X/Z,D) \ge 0$. Pick an $\R$-Cartier divisor $E \ge 0$ such that $D\sim_\R E/Z$. Then there exists a contraction $\phi: X' \to Y$ of smooth varieties such that for all sufficiently large integers $m > 0$, the rational maps $\phi_{m} : X \dashrightarrow Y_m$ given by $f ^*f_*\mathcal{O}_X(\rddown{mE})$ are birationally equivalent to $\phi$, that is, there exists a commutative diagram
	$$
	\xymatrix{
		X \ar@{-->}[d]_{\phi_m}   &  X' \ar[d]^{\phi} \ar[l]_{\pi} &\\
		Y_m  &    Y  \ar@{-->}[l]^{\varphi_m}} 
	$$  
	of rational maps $\phi_m,\varphi_m$ and a contraction $\pi$, where the horizontal maps
	are birational, $\dim Y = \kappa_\iota(X,D)$, and $\kappa(X'/Y,f^*E ) = 0$. Such a fibration is called an \emph{Iitaka fibration} of $D$. It is unique up to birational equivalence.
\end{defn}

\begin{lem}
	The definition above is well-defined and independent of the choice of $E$.
\end{lem}
\begin{proof}
	By compactification, we may assume $Z$ is projective, and hence $X,Y$ projective. The definition is well-defined by \cite[II,3.14]{nakayama}. Let $\phi':X' \to Y'$ be a relative Iitaka fibration over $Z$ associated to an $\R$-Cartier divisor $E'\ge 0$ such that $D \sim_\R E' /Z$. Pick a very general closed point $y' \in Y'$. By \cite[Proof of Lemma 2.5.6]{fujino-book}, for any sufficiently large positive integer $m$, there is an injection
	$$
	H^0(f'^{-1}(y'),\mathcal{O}_X(\rddown{mE|_{f'^{-1}(y')}})) \hookrightarrow 	H^0(f'^{-1}(y'),\mathcal{O}_X(\rddown{(m+1)E'|_{f'^{-1}(y')}}))\simeq k .
	$$
	We infer that the image of $f'^{-1}(y')$ under $\phi_m$ is a point. Therefore, by the rigidity lemma \cite[II,1.12]{nakayama}, $\phi'$ induces a birational map $\psi_m:Y' \dashrightarrow Y_m$ such that $\phi_m \circ \pi=\psi_m \circ \phi'$, which completes the proof.
\end{proof}

\noindent \textbf{Canonical models.}
Recall that, given a a proper morphism $h\colon X\to Z$ from a normal variety to a variety, an $\R$-Cartier divisor $D$ is \emph{semi-ample} over $Z$ if there exist a proper surjective morphism $g: X \to Y$ over $Z$ and an ample$/Z$ divisor $D_Y$ of $Y$ such that $D \sim_\R g^*D_Y$.

\begin{rem}[\text{\cite[Lemma 2.5]{hu2}}]\label{rem-semi-ample-divisor}
	Notation as above, let $D$ be an $\R$-Cartier divisor. 
	\begin{enumerate}
		\item $D$ is semi-ample if and only if $D$ is a convex combination of semi-ample $\Q$-divisors.
		
		\item Let $D'$ be another $\R$-Cartier divisor. If $D,D'$ are semi-ample, then so is $D+D'$.
	\end{enumerate}
\end{rem}

Given an $\R$-linear system $|D/Z|_\R$, we say a divisor $E \ge 0$ is \emph{contained in the fixed part of} $|D/Z|_\R$ if, for every $B \in |p^*D/Z|_\R$, then $B \ge E$.

\begin{defn}[\text{\cite[Definitions 3.6.5 and 3.6.7]{bchm}}]
	Let $h: X \to Z$ be a projective morphism of normal quasi-projective varieties and let $D$ be a $\R$-Cartier divisor on $X$. 
	\begin{enumerate}
		\item We say that a birational contraction $f : X \dashrightarrow X'$ over $Z$ is a \emph{semi-ample model} of $D$ over $Z$, if $f$ is $D$-non-positive, $X'$ is normal and projective over $Z$ and $D' = f_*D$ is semi-ample over $Z$.
		
		\item We say that $g : X \dashrightarrow Y$ is the \emph{ample model} of $D$ over $Z$, if $g$ is
		a rational map over $Z$, $Y$ is normal and projective over $Z$ and there
		is an ample divisor $H$ over $Z$ on $Y$ such that if $p: W \to X$ and
		$q : W\to  Y$ resolve $g$ then $q$ is a contraction morphism and we may
		write $p^*D \sim_\R q^* H + E/Z$, where $E \ge 0$ is contained in the fixed part of $ |p^*D/Z|_\R$. By \cite[Lemma 3.6.6]{bchm}, the ample model is unique up to isomorphism.
		
		\item (Canonical model.) If $(X,B)$ is a klt pair and $D=K_X+B$, then we say $Y$ is the \emph{canonical model} of $(X,B)$ over $Z$.
	\end{enumerate}
\end{defn}

\subsection{Klt-trivial fibrations}\label{subsec-lc-trivial-fib}
Recall that the discrepancy b-divisor $\mathbf{A} = \mathbf{A}(X, B)$
of a pair $(X, B)$ is the b-divisor of $X$ with the trace $ \mathbf{A}_Y$ defined by
the formula
$$
K_Y = f^*(K_X + B) +  \mathbf{A}_Y ,
$$
where $f : Y \to X$ is a proper birational morphism of normal varieties. By the definition, we have $\mathcal{O}_X (\lceil \mathbf{A}(X, B)\rceil) = \mathcal{O}_X$ when $(X, B)$
is klt (see
\cite[Lemma 3.19]{fujino-abund-saturation}).

\begin{defn}[\text{\cite[Definition 2.21]{hu2}}]\label{defn-Q-lc-trivial-fib}
	Let $\K=\Q$ or $\R$. A \emph{$\K$-klt-trivial fibration} $f : (X, B) \to
	Y$ consists of a contraction $f : X \to Y$ of normal varieties and a sub-pair $(X, B)$ satisfying the
	following properties:
	
	(1) $(X, B)$ is sub-klt over the generic point of Y;
	
	(2) $\mathrm{rank} f_*\mathcal{O}_X(\lceil  \mathbf{A}
	(X, B)\rceil) = 1$;
	
	(3) There exists an $\R$-Cartier divisor $D$ on $Y$ such that
	$$
	K_X + B \sim_\K f^*D.
	$$
\end{defn}
Notation as above, we set
\begin{align*}
b_P = \max \{ t\in \R| \text{$(X, B + tf^*P)$ is sub-lc over
	the generic point of $P$} \}
\end{align*}
and set
$$
B_Y =\sum_P (1 - b_P )P,
$$
where $P$ runs over prime divisors on $Y$. Then it is easy to see that $B_Y$ is well defined since $b_P = 1$ for all but a finite number of prime divisors and it is called the \emph{discriminant divisor}. Furthermore, we set
$$
M_Y = D- K_Y-B_Y
$$
and call $M_Y$ the \emph{moduli divisor}. Note that if $\K=\Q$, thanks to the important result \cite[Theorem 2.5]{ambro1} obtained by the theory of variations of Hodge structure, the moduli b-divisor $\mathbf{M}$ of a $\Q$-klt-trivial fibration is $\Q$-b-Cartier and b-nef. Hence $\mathbf{K}+\mathbf{B}$ is $\R$-b-Cartier.

The arguments for next lemma are taken from \cite{hu2}. 
\begin{lem}[\text{\cite[Lemma 2.22]{hu2}}]\label{lem-Q-lc-trivial-fib}
	Let $f:(X,B) \to Y$ be an $\R$-klt-trivial fibration. Then, $B$ is a convex combination of $\Q$-divisors $B_i$ such that $f:(X,B_i) \to Y$ is $\Q$-klt-trivial. Moreover, if $(X,B)$ is sub-klt, then we can choose $B_i$ so that $(X,B_i)$ is sub-klt for each $i$.
\end{lem}
\begin{proof}
	Replacing $X$ we may assume it is smooth. Let $f:(X,B) \to Y$ be an $\R$-klt-trivial fibration, $\varphi=\prod_{i=1}^k \varphi_i^{\alpha_i}$ be an $\R$-rational function so that $K_X+B+(\varphi)=f^*D$. Let $\mathcal{V} \subset \CDiv_\R(Y) $ be a finite dimensional rational linear subspace containing $D$, $\mathcal{L}\subset \CDiv_\R(X)$ be a rational polytope containing $B$ such that, for every $\Delta \in \mathcal{L}$, we have $(X,\Delta)$ is a sub-pair which is sub-klt over the generic point of $Y$. Now we consider the rational polytope 
	$$
	\mathcal{P}:=\{\Delta \in \mathcal{L}|\Delta +\sum_{i=1}^k \R(\varphi_i)\text{ intersects }f^*\mathcal{V} \}
	$$
	For every $\Delta \in \mathcal{P}$, we have further $K_X+\Delta \sim_\R 0/Y$. It is obvious that $B \in \mathcal{P}$.
	
	It suffices to show that, there exists a convex combination $B=\sum_j r_jB_j$ of $\Q$-divisors $B_j \in \mathcal{P}$ with $\mathrm{rank} f_*\mathcal{O}_X(\lceil  \mathbf{A}
	(X, B_j)\rceil) = 1$. To this end, pick a log resolution $\pi: \overline{X} \to X$ of $(X, \sum_j \Gamma_j  )$ where every element of $\mathcal{P}$ is supported by $\sum_j \Gamma_j$. Note that the proofs of \cite[Lemmas 3.19 and 3.20]{fujino-abund-saturation} are still valid for $\R$-sub-boundaries. Hence, by shrinking $Y$, we may assume $(X,\Delta)$ is sub-klt for every $\Delta \in \mathcal{P}$, and we have $$
	f_*\mathcal{O}_X(\lceil  \mathbf{A}
	(X, \Delta)\rceil)=f_* \pi_* \mathcal{O}_{\overline{X}}(\sum \lceil a_i  \rceil A_i)$$
	where $K_{\overline{X}}=\pi^*(K_X+\Delta) +\sum a_i A_i$. 
	Consider the rational sub-polytope 
	$$\mathcal{Q}=\{\Delta \in \mathcal{P}| \lceil  \mathbf{A}
	(X, \Delta)_{\overline{X}}\rceil \le \lceil  \mathbf{A}
	(X, B)_{\overline{X}}\rceil \}.$$ Then, for any $B_j \in \mathcal{Q}$, we have $\mathrm{rank} f_*\mathcal{O}_X(\lceil  \mathbf{A}
	(X, B_j)\rceil) = 1$ which completes the first assertion. The last statement is obvious.
\end{proof}

\begin{lem}\label{lem-R-klt-trivial}
	Let $f:(X,B) \to Y$ be an $\R$-klt-trivial fibration from a sub-klt pair, $B_Y$ be the discriminant divisor and $M_Y$ be the moduli divisor. Then, there exists a b-divisor $\mathbf{M}$ satisfying:
	\begin{enumerate}
		\item The trace $\mathbf{M}_Y=M_Y$.
		
		\item $(Y,B_Y+M_Y)$ is a g-sub-klt generalised pair with the moduli b-divisor $\mathbf{M}$.
	\end{enumerate}
\end{lem}
\begin{proof}
	Replacing $X$, we may assume it is smooth. By Lemma \ref{lem-Q-lc-trivial-fib}, there exists a convex combination of $B=\sum_i r_i B_i$ of $\Q$-divisors such that $f:(X,B_i) \to Y$ is $\Q$-klt-trivial. Let $\mathcal{P} \subset \CDiv_\R(X)$ be the polytope defined by $B_i$'s. For any prime divisor $P$ on $Y$, we set the function $b_P$ on $\mathcal{P}$: 
	\begin{align*}
	b_P(\Delta) = \max \{ t\in \R| \text{$(X, \Delta + tf^*P)$ is sub-lc over
		the generic point of $P$} \}.
	\end{align*}
	We note that the $b_P$ is piecewisely affine and gives a rational polyhedral decomposition of $\mathcal{P}$. Also note that there are only finitely many $P$ such that $b_P$ is not identically one on $\mathcal{P}$. Therefore, there exists a rational sub-polytope $\mathcal{Q}$ containing $B$ such that $b_P$ is affine on $\mathcal{Q}$, for any prime divisor $P$. In particular, replacing $B_i$'s and $r_i$'s, we have $B_Y=\sum_i r_i B_{Y,i}$ and $M_Y=\sum_i r_i M_{Y,i}$, where $B_Y,B_{Y,i}$ are discriminant divisors and $M_Y,M_{Y,i}$ are moduli divisors of $f:(X,B) \to Y,f:(X,B_i)\to Y$ respectively. Letting $\mathbf{M}=\sum_i r_i \mathbf{M}_i$, where $\mathbf{M}_i$ is the moduli b-divisor of $f:(X,B_i) \to Y$ for each $i$, we conclude the lemma by \cite[Theorem 2.5]{ambro1}.
\end{proof}

\begin{lem}\label{lem-rank}
	Let $f : (X, B) \to
	Y$ be a contraction of normal varieties from a klt pair $(X, B)$. Suppose $K_X+B \sim_\R R/Y$ where $R \ge 0$, and $\kappa(X/Y,R)=0$, then $$\mathrm{rank} f_*\mathcal{O}_X(\lceil  \mathbf{A}
	(X, B-R)\rceil) = 1.$$
\end{lem}
\begin{proof}
	Let $\pi:X' \to X$ be a log resolution of $(X,B)$ and write $\Delta=B-R$ and $K_{X'}=\pi^*(K_X+\Delta)+ \sum_i a_i A_i$. By \cite[Proof of Lemmas 3.19 and 3.20]{fujino-abund-saturation}, we have $ f_*\mathcal{O}_X(\lceil  \mathbf{A}
	(X, \Delta)\rceil) =(f \circ \pi)_*\mathcal{O}_X(\sum_i \lceil a_i \rceil A_i )$. Because we have
	$$\Supp \sum_i \lceil a_i \rceil A_i \subseteq \Supp \pi^*R \bigcup \ex (\pi),$$ we deduce $\kappa (X'/Y,\sum_i \lceil a_i \rceil A_i )=0$ and hence the lemma.
\end{proof}

\section{Existence of canonical models}

\begin{lem}\label{lem-ample-model}
	Let $f:X \to Y$ be a contraction of normal varieties over $Z$, $D$ be an $\R$-Cartier divisor on $X$ and $D_Y$ be an $\R$-Cartier divisor on $Y$. We suppose that:
	\begin{itemize}
		\item $D \sim_\R f^*D_Y+E/Z$ for some divisor $E \ge 0$, such that $\kappa(X/Y,E^h)=0$ and $E^v$ is very exceptional$/Y$, where $E^h$ (resp. $E^v$) denotes the horizontal (resp. vertical) part over $Y$.
		
		\item There is a semi-ample model of $D_Y/Z$.
	\end{itemize}
Then, there exists the ample model of $D/Z$.
\end{lem}
\begin{proof}
	We first reduce the lemma to the case $D_Y$ is semi-ample$/Z$. Let $\varphi:Y \dashrightarrow Y'$ be the birational contraction to a semi-ample model of $D_Y/Z$, and $p:\overline{Y}\to Y$ and $q:\overline{Y} \to Y'$ which resolve $\varphi$. We write $D_{Y'}$ for the birational transform of $D_Y$ and $p^*D_Y=q^*D_{Y'}+F$ where $F \ge 0$ is exceptional$/Y'$. Pick a resolution $\pi: \overline{X} \to X$ such that the induced map $\overline{f}:\overline{X} \dashrightarrow \overline{Y}$ is a morphism. We write $\overline{D}=\pi^*D,\overline{E}=\pi^*E+\overline{f}^*F$, and $\overline{D} \sim_\R (q \circ \overline{f})^*D_{Y'}+\overline{E}$. If we denote by $\overline{E}^h$ and $\overline{E}^v$ the horizontal and vertical part over $Y'$, then one can easily verify that $\kappa(\overline{X}/Y',\overline{E}^h)=0$, and $\overline{E}^v=\pi^*E^v+\overline{f}^*F$ is very exceptional$/Y'$. Replacing $X,Y$ with $\overline{X},Y'$ and the other data accordingly, we may assume $D_Y$ is semi-ample$/Z$.
	
	It remains to check that $E$ is contained in the fixed part of $|D/Z|_\R$. To this end, pick any $D'\sim_\R D/Z$. Since $D'|_F \sim_\R D|_F$ where $F$ is a general fbre of $f$, we have $E^h$ is contained in the fixed part of $|D/Z|_\R$. Replacing $D$ with $D-E^h$, we may assume $E$ is vertical and very exceptional$/Y$. Hence, the lemma follows from the Negativity lemma \cite[Lemma 3.3]{birkar-flip}.  
\end{proof}

\begin{rem}
	The lemma above also holds when $f$ is a proper surjetive morphism instead of a contraction. 
\end{rem}

\begin{proof}[Proof of Theorem \ref{thm-canonical-bundle-formula}]
	Since $\kappa_{\iota}(X/Y,K_X+B)=0$, by \cite[Lemma 2.10]{hashizumehu}, there exists an $\R$-Cartier divisor $D\ge 0$ such that $K_X+B \sim_\R D/Y$. Applying \cite[Theorem 2.1, Proposition 4.4]{ak}, there exist birational models $\pi:(X',\Delta') \to X$, $\phi:(Y',\Delta_{Y'}) \to Y$ such that the induced morphism $f':(X',\Delta') \to (Y',\Delta_{Y'})$ is toroidal and equidimensional to a log smooth pair. Moreover, writing $K_{X'}+B'=\pi^*(K_X+B)+E$ as in (1), by \cite[Theorem 1.1]{adk}, we have $B' \le \Delta'$ and $\Supp D' \subseteq \Delta'$ where $D'=\pi^*D+ E$. 
	Hence, there exists an $\R$-Cartier divisor $G \ge 0$, supported by $\Delta_{Y'}$, such that $D'^v -f'^*G$ is very exceptional$/Y'$, where $D'^v$ denotes the vertical$/Y'$ part. Set $R=D'-f'^*G$. We see $R$ satisfies (3). 
	
	Finally, by Lemma \ref{lem-rank}, $f':(X',\Theta) \to Y'$ is an $\R$-klt-trivial fibration, where $\Theta:=B'-R$. Hence, by Lemma \ref{lem-R-klt-trivial}, we apply a canonical bundle formula to obtain $K_{X'}+\Theta \sim_\R f'^*(K_{Y'}+B_{Y'}+M_{Y'})$, such that $(Y',B_{Y'}+M_{Y'})$ is a g-sub-klt generalised pair with the moduli b-divisor $\mathbf{M}$. It remains to check that $(Y',B_{Y'}+M_{Y'})$ is g-klt. Indeed, the effectiveness of $B_{Y'}$ follows from the construction of discriminant divisor.
\end{proof}

\begin{rem}\label{rem-lc-formula}
	Since the arguments for Lemmas \ref{lem-Q-lc-trivial-fib}, \ref{lem-R-klt-trivial} and \ref{lem-rank} are still valid for lc-trivial fibrations and lc pairs, one can easily generalise Theorem \ref{thm-canonical-bundle-formula} to lc pairs with the above argument. Note that, in this case, with notation from Theorem \ref{thm-canonical-bundle-formula}, $(Y',B_{Y'}+M_{Y'})$ is a g-lc generalised pair, and it is g-klt if all lc centres of $(X,B)$ are horizontal$/Y$.
\end{rem}

\begin{proof}[Proof of Theorem \ref{thm-canonical-model}]
	Take a relative Iitaka fibration $f:\overline{X} \to Y$ over $Z$. Replacing $(X,B)$, we may assume $X=\overline{X}$. By definition, we have $\kappa_\iota(X/Y,K_X+B)=0$. So, by a canonical bundle formula, there exists a commutative diagram
	$$
	\xymatrix{
		(X',B') \ar[d]_{f'} \ar[r]^{\pi}   &  (X,B)\ar[d]^{f}\  &\\
		Y' \ar[r]^{\phi} &    Y } 
	$$  
	which consists of birational models $\pi:X' \to X$, $\phi:Y' \to Y$, satisfying the conditions listed in Theorem \ref{thm-canonical-bundle-formula}. Replacing $(X,B),Y$ with $(X',B'),Y'$, we have $K_{X}+B \sim_{\mathbb{R}} f^*(K_{Y}+B_{Y}+M_{Y})+R$. Since $(Y,B_Y+M_Y)$ is g-klt and $K_Y+B_{Y}+M_{Y}$ is big$/Z$, $K_Y+B_Y+M_Y$ has a semi-ample model$/Z$ by \cite[Lemma 4.4(2)]{birkarzhang}. Because $R \ge 0$, $\kappa(X/Y,R^h)=0$ and $R^v$ is very exceptional$/Y$, where $R^h$ (resp. $R^v$) denotes the horizontal (resp. vertical) part over $Y$, by Lemma \ref{lem-ample-model}, we deduce that $(X/Z,B)$ has the canonical model.
\end{proof}



\begin{thebibliography}{BCHM}

\bibitem[ADK13]{adk} D.~Abramovich, J.~Denef, K.~Karu. Weak toroidalization over non-closed fields. Manuscripta Math., {\textbf{142}} (2013) (1-2), 257--271.

	

\bibitem[AK00]{ak} D.~Abramovich, K.~Karu, Weak semistable reduction in characteristic $0$, Invent. math. {\textbf{139}} (2000), no. 2, 241--273.




\bibitem[Amb04]{ambro1} F.~Ambro, Shokurov's boundary property.	J. Differential Geom. {\textbf{67}} (2004), no. 2, 229-255.













\bibitem[Bir12]{birkar-flip}
C.~Birkar, 
Existence of log canonical flips and a special LMMP, 
Publ. Math. Inst. Hautes \'Etudes Sci. {\textbf{115}} (2012), no. 1, 325--368.


\bibitem[BCHM10]{bchm}C.~Birkar, P.~Cascini, C.~D.~Hacon, J.~M\textsuperscript{c}Kernan, Existence of minimal models for varieties of log general type, J. Amer. Math. Soc. {\textbf{23}}(2010), no. 2, 405--468.








\bibitem[BZh16]{birkarzhang} C.~Birkar, D.~Q.~Zhang, Effectivity of Iitaka fibrations and pluricanonical systems of polarized pairs, Publ. Math. Inst. Hautes \'Etudes Sci. {\textbf{123}} (2016), no. 1, 283--331.
















\bibitem[Fuj12]{fujino-abund-saturation} O.~Fujino, Basepoint-free theorems: saturation, $b$-divisors, and canonical bundle formula, Algebra Number Theory, {\textbf{6}} (2012), no. 4, 797--823. 




\bibitem[Fuj-book17]{fujino-book}O.~Fujino, {\em Foundations of the minimal model program}, MSJ Mem. \textbf{35}, Mathematical Society in Japan, Tokyo, (2017). 
















\bibitem[FM00]{fujino-mori}O.~Fujino, S.~Mori, A canonical bundle formula, J. Differential Geom. {\textbf{56}} (2000), no. 1, 167--188.



























\bibitem[HH19]{hashizumehu}K.~Hashizume, Z.~Hu, On minimal model theory for log abundant lc pairs, to appear in J. Reine Angew Math.  



\bibitem[Hu20]{hu2}Z.~Hu, Log abundance of the moduli b-divisors of lc-trivial fibrations, preprint (2020), arXiv:2003.14379.  































\bibitem[Nak04]{nakayama}N.~Nakayama, {\em Zariski-decomposition and abundance}, MSJ Mem., {\textbf{14}}, Mathematical Society of Japan, Tokyo, (2004). 












\end{thebibliography}
\end{document}